\documentclass[a4paper,12pt]{article}
\usepackage{amsmath}

\newtheorem{theorem}{Theorem}
\newtheorem{corollary}[theorem]{Corollary}
\newtheorem{definition}[theorem]{Definition}

\newtheorem{proposition}[theorem]{Proposition}
\newtheorem{remark}[theorem]{Remark}
\newenvironment{proof}[1][Proof]{\noindent\textbf{#1.} }{\ \rule{0.5em}{0.5em}}

\input{tcilatex}

\begin{document}

\title{A Symmetric Algorithm for Hyperharmonic and Fibonacci Numbers}
\author{Ayhan Dil \\
%EndAName
Department of Mathematics,\\
Akdeniz University, 07058-Antalya, Turkey\\
adil@akdeniz.edu.tr\\
Istv\'an Mez\H{o}\\
Department of Algebra and Number Theory, \\
Institute of Mathematics, University of Debrecen, Hungary \\
imezo@math.klte.hu}
\date{}
\maketitle

\begin{abstract}
In this work, we introduce a symmetric algorithm obtained by the recurrence
relation $a_{n}^{k}=a_{n-1}^{k}+a_{n}^{k-1}$. We point out that this
algorithm can be apply to hyperharmonic-, ordinary and incomplete Fibonacci-
and Lucas numbers. An explicit formulae for hyperharmonic numbers, general
generating functions of the Fibonacci- and Lucas numbers are obtained.

Besides we define "hyperfibonacci numbers", "hyperlucas numbers". Using
these new concepts, some relations between ordinary and incomplete
Fibonacci- and Lucas numbers are investigated.
\end{abstract}

\section{\protect\normalsize Introduction}

The algorithm introduced below is an analog of the Euler-Seidel algorithm 
\cite{Dumont}. These kind of algorithms are useful to investigate some
recurrence relations and identities for some numbers and polynomials.

Having this concept, we give some applications for hyperharmonic numbers,
ordinary and incomplete Fibonacci and Lucas numbers.

First of all, two real initial sequences, denoted by $(a_{n})$ and $(a^{n})$%
, be given. Then the matrix $(a_n^k)$ corresponding to these sequences is
determined recursively by the formulae 
\begin{eqnarray}
a_{n}^{0} &=&a_{n},\text{ }a_{0}^{n}=a^{n}\text{, }(n\geq 0),
\label{sym-rec} \\
a_{n}^{k} &=&a_{n-1}^{k}+a_{n}^{k-1}\text{, }(n\geq 1,\,k\geq 1).  \notag
\end{eqnarray}

With induction, we get following relation which gives us any entries $%
a_{n}^{k}$ ($k$ denotes the row, $n$ is the column) in terms of the first
row's and the first column's elements: 
\begin{equation}
a_{n}^{k}=\sum_{i=1}^{k}\binom{n+k-i-1}{n-1}a_{0}^{i}+\sum_{j=1}^{n}\binom{%
k+n-j-1}{k-1}a_{j}^{0}.  \label{sym-gen.t}
\end{equation}%
By \eqref{sym-gen.t} we the get generating function of any row and column
for the matrix $(a_{n}^{k})$ (see Theorem 3). The relation \eqref{sym-gen.t}
proved to be useful for familiar sequences to investigate their structures.

There are some papers related with this work. Dumont \cite{Dumont} used
another recurrence relation which was given in \cite{Euler}, \cite{Seidel}
and he gave many applications for Bernoulli, Euler, Genocchi etc. numbers.
In \cite{Diletal}, there is a generalization of Euler-Seidel matrices for
Bernoulli, Euler and Genocchi polynomials. Present authors \cite{MD} used
Dumont's method for hyperharmonic numbers, $r$-stirling numbers and for
classification of second order recurrence sequences.

\section{\protect\normalsize Definitions and notation}

\subsection{{\protect\normalsize Euler-Seidel matrices}\textbf{.}}

Let a sequence $(a_{n})$ be given. Then the Euler-Seidel matrix
corresponding to this sequence is determined recursively by the formulae;%
\begin{eqnarray}
a_{n}^{0} &=&a_{n,}\quad (n\geq 0);  \label{c-rec} \\
a_{n}^{k} &=&a_{n}^{k-1}+a_{n+1}^{k-1},\quad (n\geq 0,\,k\geq 1).  \notag
\end{eqnarray}
The first row and column can be transformed into each other via Dumont's
identities \cite{Dumont} 
\begin{eqnarray}
a_{0}^{n} &=&\sum_{k=0}^{n}\binom{n}{k}a_{k}^{0},  \label{c-fr-fc} \\
a_{n}^{0} &=&\sum_{k=0}^{n}\binom{n}{k}(-1)^{n-k}a_{0}^{k}.  \notag
\end{eqnarray}

There is a connection between the generating functions of the initial
sequence $(a_{n})=(a_{n}^{0})$ and the generating functions of the first
column $(a_{0}^{n})$. Namely,

\begin{proposition}[Euler \protect\cite{Euler}]
Let 
\begin{equation}
a(t)=\sum_{n=0}^{\infty }a_{n}^{0}t^{n}  \label{cl-gen-row}
\end{equation}
be the generating function of the initial sequence $(a_{n}^{0})$. Then the
generating function of the sequence $(a_{0}^{n})$ is 
\begin{equation}
\overline{a}(t)=\sum_{n=0}^{\infty }a_{0}^{n}t^{n}=\frac{1}{1-t}a\left( 
\frac{t}{1-t}\right) .  \label{cl-gen-column}
\end{equation}
\end{proposition}

In the sequel, the generating functions for the columns of $(a_n^k$) will be
denoted by overline.

\subsection{\protect\normalsize \protect\bigskip Hyperharmonic numbers.}

The $n$-th harmonic number is the $n$-th partial sum of the harmonic series: 
\begin{equation*}
H_{n}=\sum_{k=1}^{n}\frac{1}{k}.
\end{equation*}
Let $H_{n}^{(1)}:=H_{n}$, and for all $r>1$ let 
\begin{equation}
H_{n}^{(r)}=\sum_{k=1}^{n}H_{k}^{(r-1)}  \label{Hnrdef}
\end{equation}
be the $n$-th hyperharmonic number of order $r$. By agreement, $%
H_{0}^{(r)}=0 $ for all $r$. These numbers can be expressed by binomial
coefficients and ordinary harmonic numbers: 
\begin{equation*}
H_{n}^{(r)}=\binom{n+r-1}{r-1}(H_{n+r-1}-H_{r-1}).
\end{equation*}

It turned out that the hyperharmonic numbers have many combinatorial
connections. To present these facts, we refer to \cite{BGG} and \cite{CG}.
Present authors gave new closed form for these numbers in \cite{MD}.

\subsection{\protect\normalsize Fibonacci and Lucas numbers.}

The sequence of the Fibonacci numbers is given by the recursion formulae 
\begin{equation*}
F_{n}=F_{n-1}+F_{n-2},\quad (n\geq 2)
\end{equation*}
with initial values $F_{0}=0,\;F_{1}=1$. The Lucas sequence $L_{n}$ has the
same recursion formulae, but $L_{0}=2,\;L_{1}=1$. The numbers $L_{n}$ and $%
F_{n}$ are connected with the formulae 
\begin{equation}
L_{n}=F_{n-1}+F_{n+1},\quad (n\geq 1).  \label{relfibandluc}
\end{equation}

One can read more on these numbers in \cite{CG}, \cite{Koshy}, \cite%
{vorobyov} and \cite{Genfuncology}. Now we cite a general generating
function for Fibonacci numbers from \cite{Koshy} (page 230) which we need
later 
\begin{equation}
\sum_{n=0}^{\infty }F_{kn+r}t^{n}=\frac{F_{r}+\left( -1\right) ^{r}F_{k-r}t}{%
1-L_{k}t+\left( -1\right) ^{k}t^{2}}.  \label{gengenfib}
\end{equation}

We can derive similar generating function for Lucas numbers easily by using (%
\ref{gengenfib}) and (\ref{relfibandluc}): 
\begin{equation}
\sum_{n=0}^{\infty }L_{kn+r}t^{n}=\frac{L_{r}+\left( -1\right) ^{r-1}L_{k-r}t%
}{1-L_{k}t+\left( -1\right) ^{k}t^{2}}.  \label{gengenluc}
\end{equation}

\subsection{\protect\normalsize Incomplete Fibonacci and Incomplete Lucas
numbers.}

The incomplete Fibonacci and incomplete Lucas numbers are defined \cite%
{flipponi} by: 
\begin{equation}
F_{n}(k)=\sum_{j=0}^{k}\binom{n-1-j}{j}\text{, }\left( n=1,2,3,...\text{; }%
0\leq k\leq \left[ \frac{n-1}{2}\right] \right) ;  \label{infibdef}
\end{equation}
\begin{equation*}
L_{n}(k)=\sum_{j=0}^{k}\frac{n}{n-j}\binom{n-j}{j},\left( n=1,2,3,...\text{; 
}0\leq k\leq \left[ \frac{n}{2}\right] \right),
\end{equation*}
where $\left[n\right] $ denotes the integer part of $n$.

The connection between ordinary and incomplete Fibonacci and Lucas numbers
are also given in \cite{flipponi} as 
\begin{equation}
F_{n}(k)=0\text{ \ \ \ \ }0\leq n\leq 2k+1,\text{ \ \ }F_{2k+1}(k)=F_{2k+1,}%
\text{ \ \ }F_{2k+2}(k)=F_{2k+2};  \label{incfib-fib1}
\end{equation}
We also need the following properties of incomplete Fibonacci and Lucas
numbers which are given in \cite{flipponi}: 
\begin{equation}
\sum_{j=0}^{h}\binom{h}{j}F_{n+j}(k+j)=F_{n+2h}(k+h)\text{, }\left( 0\leq
k\leq \frac{n-h-1}{2}\right) ,  \label{binsuminc}
\end{equation}
\begin{equation}
\sum_{j=0}^{h}\binom{h}{j}L_{n+j}(k+j)=L_{n+2h}(k+h)\text{, }\left( 0\leq
k\leq \frac{n-h}{2}\right) .  \label{binsumincL}
\end{equation}

Generating functions of these numbers are given in \cite{pinter-sri}: 
\begin{equation}
R_{k}\left( t\right) =\sum_{j=0}^{\infty }F_{j}\left( k\right) t^{j}=t^{2k+1}%
\frac{\left( F_{2k+1}+F_{2k}t\right) \left( 1-t\right) ^{k+1}-t^{2}}{\left(
1-t\right) ^{k+1}\left( 1-t-t^{2}\right) },  \label{inibgen}
\end{equation}%
\begin{equation}
S_{k}\left( t\right) =\sum_{j=0}^{\infty }L_{j}\left( k\right) t^{j}=t^{2k}%
\frac{\left( L_{2k}+L_{2k-1}t\right) \left( 1-t\right) ^{k+1}-t^{2}\left(
2-t\right) }{\left( 1-t\right) ^{k+1}\left( 1-t-t^{2}\right) }.
\label{inlucgen}
\end{equation}

\section{\protect\normalsize Generating Function of any Row and Column for
the Matrix}

After these introductory steps we are ready to formulate our results.

Now we give general terms and generating functions of any row and column of
the $(a_n^k)$ matrix using the symmetric algorithm.

\begin{proposition}
\label{gen rec prop} If (\ref{sym-rec}) holds then any entries of the matrix 
$(a_n^k)$ is
\end{proposition}

\label{pge}

\begin{equation}
a_{n}^{k}=\sum_{i=1}^{k}\binom{n+k-i-1}{n-1}a_{0}^{i}+\sum_{j=1}^{n}\binom{%
k+n-j-1}{k-1}a_{j}^{0}.  \label{genentry}
\end{equation}

\begin{proof}
Easy to prove it by considering (\ref{sym-rec}) with induction.
\end{proof}

\begin{theorem}
\label{thegengenf}Let $a_{n}^{0}$ and $a_{0}^{n}$ be two initial sequences.
Then the generating functions of the $k$th row and $n$th column of $(a_n^k)$
are 
\begin{equation}
^{k}a\left( t\right) =\sum_{n=1}^{\infty }a_{n}^{k}t^{n}=\frac{1}{\left(
1-t\right) ^{k}}\left\{ ^{0}a\left( t\right) +\frac{t}{1-t}%
\sum_{r=1}^{k}a_{0}^{r}\left( 1-t\right) ^{r}\right\} ,  \label{gensym-row}
\end{equation}
\noindent and 
\begin{equation}
^{n}\overline{a\left( t\right) }=\sum_{k=1}^{\infty }a_{n}^{k}t^{k}=\frac{1}{%
\left( 1-t\right) ^{n}}\left\{ ^{0}\overline{a\left( t\right) }+\frac{t}{1-t}%
\sum_{j=1}^{n}a_{j}^{0}\left( 1-t\right) ^{j}\right\} .  \label{gesym-column}
\end{equation}
\noindent
\end{theorem}

\begin{proof}
We prove just the first equation, the second is similar. From (\ref{genentry}%
), 
\begin{equation*}
\sum_{n=0}^{\infty }a_{n+1}^{k+1}t^{n} =\sum_{n=0}^{\infty }\left\{
\sum_{r=1}^{k+1}\binom{n+k+1-r}{n}a_{0}^{r}+\sum_{j=1}^{n+1}\binom{k+n+1-j}{%
k }a_{j}^{0}\right\} t^{n}
\end{equation*}
\begin{equation*}
=a_{0}^{1}\sum_{n=0}^{\infty }\binom{n+k}{k}t^{n}+\sum_{r=1}^{k}a_{0}^{r+1}%
\sum_{n=0}^{\infty }\binom{n+k-r}{n}t^{n}+\sum_{n=0}^{\infty
}a_{n+1}^{0}t^{n}\sum_{n=0}^{\infty }\binom{k+n}{k} t^{n}
\end{equation*}
\begin{equation*}
=\sum_{n=0}^{\infty }\binom{n+k}{k}t^{n}\left\{ a_{0}^{1}+\sum_{n=0}^{\infty
}a_{n+1}^{0}t^{n}\right\}+\sum_{r=1}^{k}a_{0}^{r+1}\sum_{n=0}^{\infty }%
\binom{n+k-r}{n}t^{n}.
\end{equation*}
Then 
\begin{equation*}
\sum_{n=1}^{\infty }a_{n}^{k+1}t^{n}=\sum_{n=0}^{\infty }\binom{n+k}{k}%
t^{n}\left\{ a_{0}^{1}t+^{0}a\left( t\right) \right\}
+\sum_{r=1}^{k}a_{0}^{r+1}t\sum_{n=0}^{\infty }\binom{n+k-r}{k-r}t^{n}.
\end{equation*}
If we write related series in terms of Newton's binomial series we get 
\begin{equation*}
\sum_{n=1}^{\infty }a_{n}^{k+1}t^{n}=\frac{1}{\left( 1-t\right) ^{k+1}}%
\left\{ ^{0}a\left( t\right) +\sum_{r=0}^{k}a_{0}^{r+1}t\left( 1-t\right)
^{r}\right\} .
\end{equation*}
The last equation gives the statement.
\end{proof}

\section{\protect\normalsize Applications}

Here we obtain some results on hyperharmonic-, ordinary Fibonacci- and Lucas
numbers using the algorithm have introduced.

\subsection{\protect\normalsize Application for Hyperharmonic Numbers}

We start with two suitable initial sequences for hyperharmonic numbers.

Let $a_{n}^{0}=\frac{1}{n+1}$ and $a_{0}^{n}=1$, $n\geq 1$ be given. If we
calculate the elements of the matrix $(a_n^k)$ with the recursive formula %
\eqref{sym-rec}, it turns out that it equals to 
\begin{equation}
\left( 
\begin{array}{ccccc}
H_{1}^{(0)} & H_{2}^{(0)} & H_{3}^{(0)} & H_{4}^{(0)} & \cdots \\ 
H_{1}^{(1)} & H_{2}^{(1)} & H_{3}^{(1)} & H_{4}^{(1)} & \cdots \\ 
H_{1}^{(2)} & H_{2}^{(2)} & H_{3}^{(2)} & H_{4}^{(2)} & \cdots \\ 
\vdots & \vdots & \vdots & \vdots & \ddots%
\end{array}%
\right)  \label{hypmat}
\end{equation}
Here $H_{n}^{(0)}=\frac{1}{n},n\geq 1.$

Now we are ready to get the well-known generating function of hyperharmonic
numbers with our method as a corollary of Theorem \ref{thegengenf}.

\begin{corollary}
We have 
\begin{equation*}
\sum_{n=1}^{\infty }H_{n}^{(k)}t^{n}=-\frac{\ln\left( 1-t\right) }{\left(
1-t\right) ^{k}}.
\end{equation*}
\end{corollary}

\begin{proof}
In Theorem \ref{thegengenf} by taking $a_{n}^{0}=\frac{1}{n+1}$ and $%
a_{0}^{n}=1$, ($n\geq 1$) one can easily get 
\begin{equation*}
^{k}a\left( t\right) =\sum_{n=2}^{\infty }H_{n}^{(k)}t^{n}=\frac{t}{\left(
1-t\right) ^{k}}\left\{ ^{0}a\left( t\right) +\frac{t}{1-t}%
\sum_{r=1}^{k}\left( 1-t\right) ^{r}\right\} .
\end{equation*}
From the identities 
\begin{equation*}
^{0}a\left( t\right) =-\frac{\ln\left( 1-t\right) }{t}-1,
\end{equation*}
and 
\begin{equation*}
\sum_{r=1}^{k}\left( 1-t\right) ^{r}=\frac{\left( 1-t\right) }{t}\left\{
1-\left( 1-t\right) ^{k}\right\} ,
\end{equation*}%
we can write, 
\begin{equation*}
\sum_{n=2}^{\infty }H_{n}^{(k)}t^{n}=\frac{t}{\left( 1-t\right) ^{k}}\left\{
-\frac{\ln\left( 1-t\right) }{t}-\left( 1-t\right) ^{k}\right\} =-\frac{%
\ln\left( 1-t\right) }{\left( 1-t\right) ^{k}}-t.
\end{equation*}
It completes the proof.
\end{proof}

Next theorem indicates relation between binomial coefficients and
hyperharmonic numbers. In \cite{BGG}, authors gave combinatorial proof of
this statement. Now we will prove by the symmetric algorithm.

\begin{theorem}
Let $n\geq 1$, $k\geq 1$. Then%
\begin{equation*}
H_{n}^{(k)}=\sum_{j=1}^{n}\binom{n+k-j-1}{k-1}\frac{1}{j}.
\end{equation*}
\end{theorem}

\begin{proof}
Let $a_{n}^{0}=\frac{1}{n+1}$ and $a_{0}^{n}=1,$ ($n\geq 1$). From (\ref%
{genentry}), 
\begin{eqnarray*}
a_{n+1}^{k+1} &=&\sum_{i=1}^{k+1}\binom{n+k-i+1}{n}+\sum_{j=1}^{n+1}\binom{%
k+n-j+1}{k}\frac{1}{j+1} \\
&=&\sum_{i=0}^{k}\binom{n+k-i}{n}+\sum_{j=0}^{n}\binom{k+n-j}{k}\frac{1}{j+2}
\\
&=&\sum_{r=0}^{k}\binom{n+r}{n}+\sum_{s=0}^{n}\binom{k+s}{k}\frac{1}{n-s+2},
\end{eqnarray*}
where $k-i=r$ and $n-j=s$. From \cite[page 160]{GKP}

\begin{equation*}
\sum_{t=a}^{b}\binom{t}{a}=\binom{b+1}{a+1}.
\end{equation*}
Hence 
\begin{equation*}
a_{n+1}^{k+1}=\binom{k+n+1}{n+1}+\sum_{s=0}^{n}\binom{k+s}{k}\frac{1}{n-s+2}%
=\sum_{s=0}^{n+1}\binom{k+s}{k}\frac{1}{n-s+2}.
\end{equation*}
Then \eqref{hypmat} yields 
\begin{equation*}
a_{n-1}^{k}=H_{n}^{(k)}=\sum_{s=0}^{n-1}\binom{k+s-1}{k-1}\frac{1}{n-s},
\end{equation*}
this completes the proof.
\end{proof}

\subsection{\protect\normalsize Applications for the Ordinary Fibonacci and
Lucas Numbers}

We point out that the symmetric algorithm is quite applicable for ordinary
Fibonacci and Lucas numbers. By starting with two different initial
sequences we get an application which gives us new identities.

Now we consider the initial sequences $a_{n}^{0}=F_{n-1}$ and $%
a_{0}^{n}=F_{2n-1}$, $n\geq 1$. This special case gives the following matrix:

\begin{equation}
\left( 
\begin{array}{ccccc}
0 & F_{0} & F_{1} & F_{2} & \cdots \\ 
F_{1} & F_{2} & F_{3} & F_{4} & \cdots \\ 
F_{3} & F_{4} & F_{5} & F_{6} & \cdots \\ 
F_{5} & F_{6} & F_{7} & F_{8} & \cdots \\ 
\vdots & \vdots & \vdots & \vdots & \ddots%
\end{array}%
\right) .  \label{fibmat}
\end{equation}
One can consider same matrix for the Lucas numbers just by substitution $%
F_{n}$ with $L_{n}$.

We prove some famous relations \cite{Koshy} for Fibonacci and Lucas numbers
with our method.

\begin{proposition}
\label{probfibeq} The following equalities hold%
\begin{equation}
F_{2n}=\sum_{i=1}^{n}F_{2i-1}\text{ and }\sum_{i=1}^{n}F_{i}=F_{n+2}-1
\label{helpfib}
\end{equation}%
and%
\begin{equation}
L_{2n}-2=\sum_{i=1}^{n}L_{2i-1}\text{ and }\sum_{i=0}^{n}L_{i}=L_{n+2}-1.
\label{helpluc}
\end{equation}
\end{proposition}

\begin{proof}
Here, we consider only (\ref{helpfib}). (\ref{helpluc}) can be proven
similarly.

For $a_{n}^{0}=F_{n-1}$ and $a_{0}^{n}=F_{2n-1}$, $n\geq 1$ we can write $%
a_{1}^{1}=F_{2},\;a_{1}^{2}=F_{4}$, and by induction $a_{1}^{n}=F_{2n}$. %
\eqref{sym-gen.t} implies that 
\begin{equation*}
a_{1}^{n}=F_{0}+\sum_{i=1}^{n}F_{2i-1}.
\end{equation*}%
These prove (\ref{helpfib}).
\end{proof}

The generating function of the first row or the first column is well-known.
We obtain generating function of any row or column and some interesting
results of them. For the sake of simplicity, let us denote the quantities

\begin{equation}
A_{n,k}:=\sum_{i=0}^{k-1}\binom{n+k-i-2}{n-1}F_{2i+1},\quad
B_{n,k}:=\sum_{i=0}^{n-1}\binom{n+k-i-2}{k-1}F_{i}.  \label{An,k}
\end{equation}

\begin{proposition}
\label{propfib1}For the values $a_{n}^{0}=F_{n-1}$ and $a_{0}^{n}=F_{2n-1}$,
($n\geq 1$), we have%
\begin{equation}
^{k}a\left( t\right) =\sum_{n=1}^{\infty }\left( A_{n,k}+B_{n,k}\right)
t^{n}=\frac{t\left\{ F_{2k}+tF_{2k-1}\right\} }{1-t-t^{2}}  \label{symgenfib}
\end{equation}%
and 
\begin{equation}
^{n}\overline{a\left( t\right) }=\sum_{k=1}^{\infty }\left(
A_{n,k}+B_{n,k}\right) t^{k}=\frac{t\left( F_{n+1}-tF_{n-1}\right) }{%
t^{2}-3t+1}.  \label{symgenfibII}
\end{equation}
\end{proposition}

\begin{proof}
From (\ref{gensym-row}), 
\begin{equation*}
^{k}a\left( t\right) =\frac{1}{\left( 1-t\right) ^{k}}\left\{ ^{0}a\left(
t\right) +\frac{t}{1-t}\sum_{r=1}^{k}F_{2r-1}\left( 1-t\right) ^{r}\right\} .
\end{equation*}%
Considering (\ref{gengenfib}), 
\begin{equation*}
^{0}a\left( t\right) =\sum_{n=1}^{\infty }F_{n-1}t^{n}=\frac{t^{2}}{1-t-t^{2}%
}\text{,}
\end{equation*}%
and 
\begin{eqnarray*}
\sum_{r=1}^{k}F_{2r-1}\left( 1-t\right) ^{r} &=&\sum_{r=1}^{\infty
}F_{2r-1}\left( 1-t\right) ^{r}-\sum_{r=k+1}^{\infty }F_{2r-1}\left(
1-t\right) ^{r} \\
&=&\frac{\left( 1-t\right) t-\left( 1-t\right) ^{k+1}\left\{ F_{2k+1}-\left(
1-t\right) F_{1-2k}\right\} }{t^{2}+t-1}.
\end{eqnarray*}%
By definition, $F_{-n}=\left( -1\right) ^{n+1}F_{n}$, thus 
\begin{equation*}
\sum_{r=1}^{k}F_{2r-1}\left( 1-t\right) ^{r}=\frac{\left( 1-t\right) \left\{
t-\left( 1-t\right) ^{k}\left( F_{2k}+tF_{2k-1}\right) \right\} }{t^{2}+t-1}.
\end{equation*}
Then 
\begin{equation*}
^{k}a\left( t\right) =\frac{1}{\left( 1-t\right) ^{k}}\left\{ \frac{t\left(
1-t\right) ^{k}\left\{ F_{2k}+tF_{2k-1}\right\} }{1-t-t^{2}}\right\} =\frac{%
t\left\{ F_{2k}+tF_{2k-1}\right\} }{1-t-t^{2}}.
\end{equation*}

The proof of (\ref{symgenfibII}) can be proven by the same approach.
\end{proof}

Let us consider a similar proposition for even and odd Fibonacci numbers.

\begin{proposition}
\label{propfib2}With initial sequences $a_{n}^{0}=F_{2n-1}$ and $%
a_{0}^{n}=F_{2n}$, $n\geq 1$, we have%
\begin{equation*}
\sum_{n=1}^{\infty }\left( C_{n,k}+A_{k,n}\right) t^{n}=\frac{-t}{t^{2}+t-1}%
\left\{ \frac{t\left( t^{2}-t+1\right) }{\left( 1-t\right) ^{k}\left(
t^{2}-3t+1\right) }+F_{2k+1}+tF_{2k}\right\}
\end{equation*}%
and%
\begin{equation*}
\sum_{k=1}^{\infty }\left( C_{n,k}+A_{k,n}\right) t^{k}=\frac{t}{t^{2}+t-1}%
\left\{ \frac{2t\left( t^{2}-t+1\right) }{\left( 1-t\right) ^{n}\left(
t^{2}-3t+1\right) }-F_{2n}-tF_{2n-1}\right\} ,
\end{equation*}%
where 
\begin{equation*}
C_{n,k}:=\sum_{i=0}^{k-1}\binom{n+k-i-2}{n-1}F_{2i}.
\end{equation*}
\end{proposition}

\begin{remark}
By considering Proposition \ref{probfibeq}, Proposition\ref{propfib1} and
Proposition\ref{propfib2} we have the generating functions for tails of the
Fibonacci sequence.
\end{remark}

\begin{remark}
We obtain similar propositions to proposition \ref{propfib1} and proposition %
\ref{propfib2} for the Lucas numbers just by changing $F_{n}$ with $L_{n}$.
\end{remark}

\subsection{Applications for The Incomplete Fibonacci and Incomplete Lucas
Numbers}

We have applications with two different methods.

\subsubsection{With Euler-Seidel Algorithm}

We give some applications on incomplete Fibonacci numbers with Euler-Seidel
method.

First, take the incomplete Fibonacci numbers $F_{r+n}\left( s+n\right) $ as $%
a_{n}^{0}$. From ($\ref{c-fr-fc}$), 
\begin{equation*}
a_{0}^{n}=\sum_{k=0}^{n}\binom{n}{k}F_{r+k}\left( s+k\right) .
\end{equation*}
($\ref{binsuminc}$) implies 
\begin{equation*}
a_{0}^{n}=F_{r+2n}\left( s+n\right) .
\end{equation*}

Because of the selection of $a_{n}^{0}$ and the last equation of $a_{0}^{n}$%
, we obtain the dual formulae of (\ref{binsuminc}): 
\begin{equation}
F_{r+n}\left( s+n\right) =\sum_{k=0}^{n}\binom{n}{k}\left( -1\right)
^{n-k}F_{r+2k}\left( s+k\right) \text{, \ \ }0\leq s\leq \frac{r-n-1}{2}.
\label{inc-cessaro}
\end{equation}
Similarly, 
\begin{equation*}
L_{r+2n}\left( s+n\right) =\sum_{k=0}^{n}\binom{n}{k}L_{r+k}\left(
s+k\right) \text{, \ \ }0\leq s\leq \frac{r-n}{2},
\end{equation*}
and its dual is 
\begin{equation*}
L_{r+n}\left( s+n\right) =\sum_{k=0}^{n}\binom{n}{k}\left( -1\right)
^{n-k}L_{r+2k}\left( s+k\right) \text{, \ \ }0\leq s\leq \frac{r-n}{2}.
\end{equation*}

Secondly, let $a_{n}^{0}=F_{n}\left( k\right)$. Then from ($\ref{c-fr-fc}$), 
\begin{equation*}
a_{0}^{n}=\sum_{l=0}^{n}\binom{n}{l}F_{l}\left( k\right) .
\end{equation*}
We present a new formula to this quantity.

\begin{theorem}
\begin{equation*}
\sum_{l=0}^{n}\binom{n}{l}F_{l}\left( k\right) =\left\{ 
\begin{array}{ll}
0 & \text{if }n<2k+1 \\ 
F_{2k+1} & \text{if }n=2k+1 \\ 
F & \text{if }n\geq 2k+2%
\end{array}
\right.
\end{equation*}
where, 
\begin{eqnarray*}
F&=&\sum_{r=2k+1}^{n}\left[ F_{2k}\binom{r}{2k}+F_{2k-1}\binom{r-1}{2k-1} %
\right] F_{2n-2r} \\
&&-\sum_{r=0}^{n}\sum_{m=0}^{r}F_{2n-2r-4k-2}\binom{r+k-m-1}{k}\binom{m+k}{k}%
2^{m}.
\end{eqnarray*}
\end{theorem}

\begin{proof}
\eqref{inibgen} gives that 
\begin{equation*}
a\left( t\right) =\sum_{j=0}^{\infty }F_{j}\left( k\right) t^{j}=t^{2k+1}%
\frac{\left( F_{2k+1}+tF_{2k}\right) \left( 1-t\right) ^{k+1}-t^{2}}{\left(
1-t\right) ^{k+1}\left( 1-t-t^{2}\right) }.
\end{equation*}
From (\ref{cl-gen-column}), 
\begin{eqnarray*}
\overline{a}(t) &=&\sum_{n=0}^{\infty }\left[ \sum_{l=0}^{n}\binom{n}{l}%
F_{l}\left( k\right) \right] t^{n}=\frac{t^{2k+1}}{\left( 1-2t\right)
^{k+1}\left( t^{2}-3t+1\right) \left( 1-t\right) ^{k-1}} \\
&&\times \left\{ \left( F_{2k+1}-tF_{2k}\right) \frac{\left( 1-2t\right)
^{k+1}}{\left( 1-t\right) ^{k+2}}-\frac{t^{2}}{\left( 1-t\right) ^{2}}%
\right\} .
\end{eqnarray*}
By taking out the generating function of even Fibonacci numbers we have 
\begin{eqnarray*}
\overline{a}(t) &=&\frac{t^{2k+1}}{\left( t^{2}-3t+1\right) }\left\{ \left(
F_{2k+1}-tF_{2k}\right) \sum_{n=0}^{\infty }\binom{n+2k}{n}t^{n}\right. \\
&&\left. -t^{2}\sum_{n=0}^{\infty }\binom{n+k}{n}2^{n}t^{n}\sum_{n=0}^{%
\infty }\binom{n+k}{n}t^{n}\right\} \\
&=&t^{2k}\sum_{n=0}^{\infty }\sum_{r=0}^{n}F_{2n-2r}\left\{ F_{2k+1}\binom{%
r+2k}{r}-tF_{2k}\binom{r+2k}{r}\right. \\
&&\left. -t^{2}\left[ \sum_{m=0}^{r}\binom{r-m+k}{r-m}\binom{m+k}{m}2^{m}%
\right] \right\} t^{n} \\
&=&F_{2}F_{2k+1}t^{2k+1}+\sum_{n=2k+2}^{\infty }\left\{
\sum_{r=0}^{n-2k}F_{2n-4k-2r}F_{2k+1}\binom{r+2k}{r}\right. \\
&&-\sum_{r=1}^{n-2k}F_{2n-4k-2r}F_{2k}\binom{r+2k-1}{r-1} \\
&&\left. -\sum_{r=2}^{n-2k}\sum_{m=0}^{r-2}F_{2n-4k-2r}\binom{r-2-m+k}{r-2-m}%
\binom{m+k}{m}2^{m}\right\} t^{n}.
\end{eqnarray*}
After some rearrangement, 
\begin{equation*}
\overline{a}(t) =F_{2}F_{2k+1}t^{2k+1}+\sum_{n=2k+2}^{\infty }\left\{
\sum_{r=2}^{n-2k}F_{2n-4k-2r}\left[ F_{2k+1}\binom{r+2k}{r}\right. \right.
\end{equation*}
\begin{equation*}
\left. -F_{2k}\binom{r+2k-1}{r-1}-\sum_{m=0}^{r-2}\binom{r-2-m+k}{r-2-m} 
\binom{m+k}{m}2^{m}\right]
\end{equation*}
\begin{equation*}
\left. +F_{2n-4k}F_{2k+1}+\left( 2k+1\right)
F_{2n-4k-2}F_{2k+1}-F_{2n-4k-2}F_{2k}\right\} t^{n}
\end{equation*}
\begin{equation*}
=F_{2}F_{2k+1}t^{2k+1}+\sum_{n=2k+2}^{\infty }\left\{
\sum_{r=2}^{n-2k}F_{2n-4k-2r}\left[ \binom{r+2k-1}{r-1}\frac{%
rF_{2k}+2kF_{2k+1}}{r}\right. \right.
\end{equation*}
\begin{equation*}
\left. -\sum_{m=0}^{r-2}\binom{r-2-m+k}{r-2-m}\binom{m+k}{m}2^{m}\right]
\end{equation*}
\begin{equation*}
\left. +F_{2n-4k}F_{2k+1}+\left( 2k+1\right)
F_{2n-4k-2}F_{2k+1}-F_{2n-4k-2}F_{2k}\right\} t^{n}
\end{equation*}
we proved the theorem.
\end{proof}

A same approach proves a parallel result for incomplete Lucas numbers

\begin{theorem}
\begin{equation*}
\sum_{l=0}^{n}\binom{n}{l}L_{l}\left( k\right) =\left\{ 
\begin{array}{ll}
0 & \text{if }n<2k \\ 
L_{2k} & \text{if }n=2k \\ 
\left( 2k+1\right) L_{2k}+L_{2k+2} & \text{if }n=2k+1 \\ 
L & \text{if }n\geq 2k+2%
\end{array}%
\right.
\end{equation*}%
where,%
\begin{eqnarray*}
L &=&\sum_{r=0}^{n-2k-2}\left( \left\{
F_{2n-4k-2r+2}L_{2k}-F_{2n-4k-2r}L_{2k-2}\right\} \binom{r+2k-1}{r}\right. \\
&&-\left\{ F_{2n-4k-2r-5}+F_{2n-4k-2r-3}\right\} \sum_{m=0}^{r}\binom{r-m+k}{%
r-m}\binom{m+k}{m}2^{m} \\
&&+L_{2k}\binom{n-1}{n-2k}+L_{2k+2}\binom{n-2}{n-2k-1}.
\end{eqnarray*}
\end{theorem}

\subsubsection{An Application of the With Symmetric Algorithm}

Here we give new concepts as "hyperfibonacci numbers" and "hyperlucas
numbers" like "hyperharmonic numbers". They will be useful for us.

\begin{definition}
Let the hyperfibonacci numbers $F_{n}^{\left( r\right) }$ and the hyperlucas
numbers $L_{n}^{\left( r\right) }$be defined respectively as 
\begin{eqnarray}
F_{n}^{\left( r\right) } &=&\sum_{k=0}^{n}F_{k}^{\left( r-1\right) }\text{,
with }F_{n}^{\left( 0\right) }=F_{n},\text{ }F_{0}^{\left( r\right) }=0\text{%
, and }F_{1}^{\left( r\right) }=1;  \label{defhypfib} \\
L_{n}^{\left( r\right) } &=&\sum_{k=0}^{n}L_{k}^{\left( r-1\right) }\text{,
with }L_{n}^{\left( 0\right) }=L_{n},\text{ }L_{0}^{\left( r\right) }=0\text{%
, and }L_{1}^{\left( r\right) }=1.  \notag
\end{eqnarray}
\end{definition}

\begin{proposition}
We have generating functions of the hyperfibonacci numbers and the
hyperlucas numbers, respectively, as follows 
\begin{equation*}
\sum_{n=0}^{\infty }F_{n}^{\left( r\right) }t^{n}=\frac{t}{\left(
1-t-t^{2}\right) \left( 1-t\right) ^{r}},
\end{equation*}%
\begin{equation*}
\sum_{n=0}^{\infty }L_{n}^{\left( r\right) }t^{n}=\frac{2-t}{\left(
1-t-t^{2}\right) \left( 1-t\right) ^{r}}.
\end{equation*}
\end{proposition}

\begin{proof}
Proof is obtained immediately by using Cauchy product and induction.
\end{proof}

Now we are ready for the application. Let us recall \eqref{An,k}. By $\left( %
\ref{inibgen}\right) $, 
\begin{equation*}
\sum_{j=0}^{\infty }F_{j}\left( k\right) t^{j}=t^{2k}\sum_{n=1}^{\infty
}\left( A_{n,k}+B_{n,k}\right) t^{n}-\frac{t^{2k+2}}{\left( 1-t\right) ^{k+1}%
}\sum_{n=0}^{\infty }F_{n}t^{n}.
\end{equation*}%
Applying the concept of hyperfibonacci numbers, we can rewrite this as 
\begin{equation*}
\sum_{j=0}^{\infty }F_{j}\left( k\right) t^{j}=\sum_{n=2k+1}^{\infty }\left(
A_{n-2k,k}+B_{n-2k,k}\right) t^{n}-\sum_{n=2k+2}^{\infty }F_{n-2k-2}^{\left(
k+1\right) }t^{n}
\end{equation*}

Here with help of the proposition \ref{probfibeq} we have the following
theorem.

\begin{theorem}
\label{lastfib}We have 
\begin{equation}
F_{n}\left( k\right) =\left\{ 
\begin{tabular}{ll}
$0$ & $\text{if }0\leq n<2k+1$ \\ 
$F_{2k+1}$ & $\text{if }n=2k+1$ \\ 
$A_{n-2k,k}+B_{n-2k,k}-F_{n-2k-2}^{\left( k+1\right) }$ & $\text{if }n>2k+1$%
\end{tabular}%
\right.  \label{theoAn,khyp}
\end{equation}
\end{theorem}

Theorem \ref{lastfib} provides an interesting corollary.

\begin{corollary}
For ordinary Fibonacci numbers, the following equalities are valid 
\begin{equation}
F_{2k+1}-1=\sum_{i=0}^{k-1}\left( k-i\right) F_{2i+1},  \label{fibnew1}
\end{equation}
\begin{equation}
F_{2k+2}-k-1=\sum_{i=0}^{k-1}\binom{k+1-i}{2}F_{2i+1}.  \label{fibnew2}
\end{equation}
\end{corollary}

\begin{proof}
Lets take $n=2k+2$ in $\left( \ref{theoAn,khyp}\right) $ and $n=2k+3$ in ($%
\ref{theoAn,khyp}$).
\end{proof}

The similar result for incomplete Lucas numbers is

\begin{theorem}
\label{lastluc}We have 
\begin{equation}
L_{n}\left( k\right) =\left\{ 
\begin{tabular}{ll}
$0$ & $\text{if }0\leq n<2k$ \\ 
$L_{2k}$ & $\text{if }n=2k$ \\ 
$A_{2,k}+B_{2,k}$ & $\text{if }n=2k+1$ \\ 
$A_{n-2k+1,k}+B_{n-2k+1,k}-L_{n-2k-2}^{\left( k+1\right) }$ & $\text{if }%
n\geq 2k+2$%
\end{tabular}
\right.  \label{incompalg}
\end{equation}
\end{theorem}

\begin{proof}
Proof is similar to the theorem \ref{lastfib}.
\end{proof}

\begin{corollary}
We have 
\begin{equation*}
L_{2k+1}=\sum_{i=0}^{k-1}\left( k-i\right) L_{2i+1}+2k+1.
\end{equation*}
\end{corollary}

\begin{center}
\textbf{Acknowledgement}
\end{center}

The authors would like to thank to Professor Akos Pinter for his help and
support.


\begin{thebibliography}{10}
\bibitem[1]{BGG} Benjamin, A. T.; Gaebler, D. and Gaebler, R. \emph{A
Combinatorial Approach to Hyperharmonic Numbers}, Integers: The Electornic
Journal of Combinatorial Number Theory, Vol 3(2003), p. 1-9, \#A15.

\bibitem[2]{CG} Conway, J. H. and Guy, R. K. \emph{The Book of Numbers}, New
York, Springer-Verlag, 1996.

\bibitem[3]{Diletal} Dil, A.; Kurt, V. and Cenkci, M. \emph{Algorithms for
Bernoulli and Allied Polynomials}, J. Integer Seq. Vol. 10 (2007) Article
07.5.4.

\bibitem[4]{Dumont} Dumont, D. \emph{Matrices d'Euler-Seidel}, Seminaire
Lotharingien de Combinatorie, 1981, B05c.

\bibitem[5]{Euler} Euler, L. \emph{De Transformatione Serierum}, Opera
Omnia, series prima, Vol. X, Teubner, 1913.

\bibitem[6]{GKP} Graham, R. L.; Knuth, D. E. and Patashnik, O. \emph{%
Concrete Mathematics}, Addison Wesley, 1993.

\bibitem[7]{flipponi} Filipponi P., \emph{Incomplete Fibonacci and Lucas
Numbers, }Rend. Circ. Mat. Palermo 45 (1996), 37-56

\bibitem[8]{Koshy} Koshy, T. \emph{Fibonacci and Lucas Numbers with
Applications}, John Wiley and Sons, 2001.

\bibitem[9]{MD} Mez\H{o}, I.; Dil, A. \emph{Euler-Seidel Algorithm for
Hyperharmonic Numbers, r-stirling Numbers, and Reccurences}, submitted

\bibitem[10]{pinter-sri} Pint\'er A., Srivastava H. M., \emph{Generating
Functions of the Incomplete Fibonacci and Lucas Numbers, }Rend. Circ. Mat.
Palermo Serie II. Tomo XLVIII (1999), 591-596

\bibitem[11]{Seidel} Seidel, L. \emph{\"{U}ber Eine Einfache Enstehung Weise
der Bernoullischen Zahlen und Einiger Verwandten Reihen}, Sitzungsberichte
der M\"{u}nch. Akad. Math. Phys. Classe (1877), p. 157-187.

\bibitem[12]{vorobyov} Vorobyov, N. N. \emph{The Fibonacci Numbers}, D. C.
Heath and Company, Boston,1963

\bibitem[13]{Genfuncology} Wilf, H. S. \emph{Generatingfunctionology},
Academic Press, 1994.
\end{thebibliography}
\end{document}